\documentclass{amsart}
\usepackage{amssymb}
\usepackage{amsfonts}
\usepackage{amsmath}
\usepackage{latexsym}
\usepackage{hyperref}
\usepackage{mathtools}
\usepackage{tikz, float}
\usepackage[british]{babel}
\usepackage{enumerate}

\usepackage[all]{xy}

\newtheorem{theorem}{THEOREM}[section]
\newtheorem{lemma}[theorem]{LEMMA}
\newtheorem{corollary}[theorem]{COROLLARY}
\newtheorem{proposition}[theorem]{PROPOSITION}
\newtheorem{remark}[theorem]{REMARK}
\newtheorem{definition}[theorem]{DEFINITION}

\makeatletter
\let\orgdescriptionlabel\descriptionlabel
\renewcommand*{\descriptionlabel}[1]{%
  \let\orglabel\label
  \let\label\@gobble
  \phantomsection
  \edef\@currentlabel{#1}%
  \let\label\orglabel
  \orgdescriptionlabel{#1}%
}
\makeatother

 \def\set#1{{\{ #1\}}}
 
 \def\lr#1{{\langle #1 \rangle}}
 
  \def\nb#1{}
  
  \def\nodes{{\sf nodes}}
  
   \def\c#1{{\mathcal #1}}
 
 \def\y{{\sf y}}
 \def\g{{\sf g}}
 \def\r{{\sf r}}
 \def\bb{{\sf b}}% \b mean bold
 \def\w{{\sf w}}

 \def\conv{\smile}
\DeclareMathOperator{\comp}{;}

 \def\G{{\bf G}}

 \def\ef{Ehrenfeucht Fra\"iss\'e}
\title{First-order axiomatisations of representable relation algebras need formulas of unbounded quantifier depth}
\author{Rob Egrot} 
\author{Robin Hirsch}
\subjclass[2020]{Primary 	03G15; Secondary 05C90}
\keywords{RRA, Representable relation algebras, Finite variable axiomatisation, Bounded quantifier depth axiomatisation}

\begin{document}

\begin{abstract}
Using a variation of the rainbow construction and various pebble and colouring games, we prove that RRA, the class of all \emph{representable relation algebras},  cannot be axiomatised by any first-order  relation algebra theory of bounded quantifier depth.    We also prove that the class At(RRA) of atom structures of representable, atomic  relation algebras cannot be defined by any set of sentences in the language of RA atom structures that uses only a finite number of variables.

\end{abstract}
\maketitle

\section{Introduction}   
Relation algebras were introduced by Tarski and various coworkers during the 1940s, with the modern definition of a relation algebra being given in \cite{JT48, ChiTar51}. Tarski provided some motivation for this project in the earlier \cite{T41}, where he noted that though the theory of binary relations was by this time of `universally recognized' significance, it was then no more developed than it had been at the end of the 19th century. Much more on the historical development of the calculus of relations and relation algebras can be found in \cite{Ma91}. Relation algebras are of theoretical interest as they provide an elegant formalism for the calculus of relations, which is itself an adequate foundational framework for set theory, and thus for mathematics itself (see \cite{TG87} for the extensive details, and \cite{Giv88,Giv06} for readable summaries). Relation algebraic methods have also been useful in proving metamathematical results, for example that for all $n\geq3$ there are sentences involving only 3 variables whose formal proofs require $n$ variables \cite{HHM98:jsl}. In addition, relation algebras and their generalisations have numerous practical applications in computer science, for example in verification \cite{Gut18}, computation tasks involving finite topologies \cite{BW17}, and navigation of XML documents \cite{FGPVGW16}, to name just a few.

A relation algebra as per Tarski's formulation is an algebraic structure axiomatised by a certain finite set of equations (see the next section for details). In other words, the class of relation algebras is a finitely based variety, which is denoted RA. As discussed above, RA emerged as an attempt to capture properties of binary relations, so it is natural to ask whether this attempt was successful. The answer turns out to be `not entirely', because there are algebras in RA which do not correspond to concrete systems of binary relations \cite{Ly50}. These algebras are said to not be \emph{representable}, and we provide a formal account of what exactly this means in the next section. 

So Tarski's axiomatisation of RA was in a sense a failure, but it also turns out to be a remarkable success, for the following reason. Relation algebra equations correspond exactly with first-order statements about binary relations that can be stated using at most 3 variables (see \cite[theorems 3.9(viii)(ix)]{TG87} for a proof), and, moreover, a relation algebra equation is valid in RA if and only if the corresponding first-order sentence is \emph{provable} (in classical proof systems) using at most 4 variables \cite[theorem 24]{Ma89}. So Tarski's axioms neatly capture an intuitively meaningful fragment of the first-order theory of binary relations. Furthermore, an axiomatisation such that the `4' in the statement above could be replaced by `5' would require an infinite number of additional axioms (this result seems to lack a precise statement in the literature, but it can be pieced together from the material in e.g. \cite[section 6]{HH:raca2}).   

Nevertheless, from the perspective of capturing the true properties of binary relations, the class of main interest is that of the representable relation algebras (RRA). Unfortunately, RRA is a difficult class to axiomatise. While it is a variety \cite{T55}, and is even definable by a recursively enumerable equational theory (see  \cite[Theorem~8.4]{HH:book} for an example of such an equational theory), it is known that no finite set of relation algebra formulas can define it \cite{Mon64}. Indeed, no set (even  infinite)  of equations with finitely many variables can define it \cite[theorem~3.5.6]{Jon88}, and any axiomatisation of RRA must involve infinitely many non-canonical formulas \cite{HVCan}. 

In this paper, we extend  Monk and J\'onsson's  negative results by showing that there is no upper bound on the depth of quantifiers needed in an axiomatisation of RRA. The strategy of the paper is based on the well known fact that an atomic relation algebra can be defined from its set of atoms by describing how they interact with each other (this is known as providing an \emph{atom structure}). Since complicated relation algebras can be generated from relatively simple atom structures, this is often an efficient way to construct relation algebras with desirable properties.  

Here our goals are achieved by constructing, for each $n<\omega$, two relation algebras, $\c A_n$ and $\c B_n$,  with one representable but not the other, that cannot be distinguished by any relation algebra formula of quantifier-depth at most $n$. For this construction we use a novel variation of the `rainbow construction' introduced in \cite{Hi94c} (another version of this construction is discussed in \cite[chapter 16]{HH:book}).  Background on atom structures can be found in the next section, and the definition of the rainbow construction to be used here is given in section \ref{S:rainbow}.  

Another advantage of atom structures is that games played on what are known as \emph{atomic networks} can be used to show the associated atomic relation algebra is completely representable. We take advantage of this in section~\ref{S:rainbow} to describe the circumstances in which the algebra constructed from the modified rainbow construction is representable (theorem \ref{thm:rainbow}). As an immediate corollary to this we are able to prove that the class of atom structures of representable atomic relation algebras cannot be defined by any set of sentences in the language of RA atom structures that uses only a finite set of variables (corollary \ref{C:atoms}).

Later, in section \ref{S:axioms}, for $n\in\omega$ we find simple sufficient conditions for two algebras constructed using our rainbow construction to be equivalent with respect to first-order sentences of quantifier depth $n$. The methods here also use games. First a straightforward adaptation of the well known Ehrenfeucht-Fra\"iss\'e pebble game to relation algebras, and also a new colouring game which we call the \emph{Seurat game for sets}, as it is a special case of the Seurat games for directed graphs and binary relational structures introduced in \cite{EgrHirgames,EgrHirNote}. 

As mentioned previously, the main result (theorem \ref{thm:main}) is then obtained simply by, for each $n\in \omega$, describing the algebras $\c A_n$ and $\c B_n$. The definition of the modified rainbow construction together with the characterisation of representability provided by \ref{thm:rainbow} and the sufficient condition for quantifier depth $n$ equivalence provided by corollary \ref{C:K-equiv} make checking that the algebras have the required properties straightforward.

We had originally hoped to include in this paper a proof that no theory involving only finitely many variables (but with potentially unbounded quantifier depth) could define RRA.    An obvious approach would again be to construct (for $n<\omega$) two relation algebras, one in RRA one not, that cannot be distinguished by any $n$-variable formula.  However, many relation algebra properties may be expressed with a finitely bounded number of variables. For example, with just two reusable variables we can write a formula that  holds precisely on those relation algebras that have a given finite cardinality (see the discussion at the end of section \ref{S:rainbow}). The algebras $\c A_n, \c B_n$ used in this paper for the quantifier depth result as described above have different finite cardinalities, and so they are distinguishable with just two variables, making them unsuitable for the task. 

It remains unknown whether a finite variable axiomatisation of RRA exists.    A suggestion for how to approach and perhaps solve this problem was given  in  \cite[Problem 1]{HH:book}, however the suggested approach was faulty. An early draft of this paper contained a correction to the strategy from \cite{HH:book}. This correction involves a more complex rainbow construction, a version of the Seurat game for general binary relational structures, and a significantly more complicated version of theorem \ref{thm:rainbow}. Since the additional technical overhead obscures the argument for the main results here, and since there are independent reasons to believe the corrected strategy might be very difficult, if not impossible, we have moved this material and more detailed discussion to \cite{EgrHirNote}.

\section{Technical background}
\subsection{Relation Algebra, Representation, Complete Representation}\label{S:reps}
A \emph{relation algebra} $\c A=(A, 0, 1, +, -, 1', {}^\conv, ;)$ consists of a set $A$ with elements $0, 1, 1'$, unary functions $-, {}^\smile$ and binary functions $+, ;$ over $A$, such that
\begin{enumerate}[(i)]
\item $(A, 0, 1, +, -)$  is a boolean algebra, 
\item $(A, 1', {}^\smile, ;)$ is a convoluted monoid,  
\item ${}^\smile$  and $;$ are normal operators, and
\item   the Piercean law $a^\conv;(-(a;b))\leq -b$ holds.  
\end{enumerate}
Since these axioms are all equations, the class ${\bf RA}$ of all relation algebras is a finitely based equational variety.   

Let $E$ be any equivalence relation over base $X$.  The \emph{proper relation algebra} $\c P(E)=(\wp(E), \emptyset, E, \cup, \setminus_E, Id_X, {}^\smile, ;)$ consists  of all subsets of $E$, with the identity $Id_X=\set{(x, x):x\in X}$, converse defined by $a^\conv=\set{(y, x):(x, y)\in a}$, and composition defined by $a;b=\set{(x, y): \exists z\; ((x, z)\in a\wedge (z, y)\in b}$, where $a, b \subseteq E$.   The axioms of relation algebra are  mostly easy to verify in a proper relation algebra, perhaps the Peircean law would take slightly longer.

An embedding $\theta:\c A\rightarrow\c P(E)$ from an abstract relation algebra $\c A$ into a proper relation algebra $\c P(E)$ is called a \emph{representation} of $\c A$.  For $a\in\c A$ we write $a^\theta$ for the image of $a$ under $\theta$, which is a binary relation.   The class of all representable relation algebras (RRA) turns out to be a proper subclass of RA (not every relation algebra is representable).    If a representation $\theta$ preserves arbitrary suprema wherever they exist it is called a \emph{complete representation}.  For finite relation algebras, every representation is trivially complete, as there are no non-finite suprema, but this is not true in general. 

For each $(x, y)\in 1^\theta$, the set $\set{a\in\c A: (x, y)\in a^\theta}$ is an ultrafilter of the boolean part of $\c A$.    An ultrafilter is \emph{principal} if it includes its infimum, and is \emph{non-principal} otherwise.  In the former case the infimum must be an \emph{atom} (minimal non-zero element of $\c A$), in the latter case the infimum can be shown to be zero.  A representation $\theta$ is \emph{atomic} if for all $(x, y)\in 1^\theta$  there is an atom $a$ such that $(x, y)\in a^\theta$.  By \cite[theorem 2.21]{HH:book}  a representation is complete if and only if it is atomic.     So, if $\theta$ is a complete representation and $(x, y)\in 1^\theta$, there is a (necessarily unique) atom $a$ such that $(x, y)\in a^\theta$. We denote this atom by $(x, y)_\theta$, so $(x, y)\in ((x,y)_\theta)^\theta$, for $(x, y)\in 1^\theta$.

\subsection{Atom structure, Peircean Transforms}
A boolean algebra (possibly with additional operations, e.g. a relation algebra) is \emph{atomic} if every non-zero element is above an atom. This is equivalent to saying that every element is the supremum of the set of atoms below it.  Every finite boolean algebra is atomic.  Using the relation algebra axioms, it can be shown that that the operators ${}^\smile, ;$ are \emph{completely additive}, i.e. if the supremum $\bigvee S$ of a subset $S$ of a relation algebra exists then $(\bigvee S)^\conv$ is the supremum of the set  $\set{s^\conv:s \in S}$,   \/ and for any $a$ in the algebra, $a;\bigvee S$ is the supremum of $\set{a;s: s\in S}$ and $\bigvee S;a$ is the supremum of $\set{s;a:s\in S}$ (see e.g. \cite[p109]{HH:book} for a proof). 

It follows easily from completeness of $^\smile$ and relation algebra condition (ii) that $x^\smile$ is an atom if and only if $x$ is.  Hence, for an atomic relation algebra $\c A$, each operator is determined by its restriction to atoms.  This information is conveyed by its \emph{atom structure}, which consists of the set of atoms, the set of atoms below the identity, the set of pairs $(a, a^\conv)$ where $a$ ranges over atoms, and the set of  \emph{consistent} triples of atoms $(a, b, c)$ where $a;b\geq c$. [Beware, \cite{HH:book} uses the consistency condition $a;b\geq c^\conv$, so $(a, b, c)$ is consistent according to our definition here if and only if $(a, b, c^\conv)$ is consistent by the \cite{HH:book} definition.] 
%This is equivalent because it follows from the relation algebra axioms that for all atoms $a,b,c$ we have $a;b\geq c^\conv \iff b^\conv;a^\conv \geq c$. So $(a,b,c)$ is consistent according to the \cite{HH:book} definition if and only if $(b^\conv,a^\conv, c)$ is consistent according to the definition we use here.] 
 Instead of  giving the set of consistent triples of atoms it is often convenient to specify the 
\emph{forbidden} triples of atoms $(a, b, c)$ --- those where $a;b\;\cdot\; c=0$.  

A consequence of the relation algebra axioms, in particular the Peircean law, is that if  $(a, b, c)$ is a forbidden triple of atoms then all six so-called \emph{Peircean transforms} of $(a, b, c)$,  i.e.  
\[(a, b, c),(a^\conv, c, b), (c, b^\conv, a), (b, c^\conv, a^\conv), (c^\conv, a, b^\conv), (b^\conv, a^\conv, c^\conv),\]
are also are forbidden.

Every atomic relation algebra embeds into the \emph{complex algebra}  of its atom structure, a uniquely determined relation algebra whose elements are arbitrary sets of atoms  (see  \cite[section 2.7.2]{HH:book}), and a finite relation algebra is isomorphic to the complex algebra of its atom structure.   It is often convenient to define a relation algebra by giving an atom structure, i.e. a set of atoms, the subset of atoms below the identity, the consistent/forbidden triples and the converses. This is what we will do in section \ref{S:rainbow}, for example.
\subsection{Atomic Networks, Games}
Given an atomic relation algebra $\c A$ with atoms $At(\c A)$, an atomic labelling $N=(nodes(N), \lambda)$ consists of a set of nodes and an edge labelling function $\lambda: nodes(N)\times nodes(N)\rightarrow At(\c A)$.  If
\begin{enumerate}[(i)]
\item $\lambda(x, x)\leq 1'$, 
\item $\lambda(y, x)=(\lambda(x, y))^\conv$, and
\item $(\lambda(x, y), \lambda(y, z), \lambda(x, z))$ is not forbidden, for all $x, y, z\in nodes(N)$, 
\end{enumerate}
the atomic labelling is \emph{coherent} and we call it an \emph{atomic network}.  Given atomic labellings $N=(nodes(N), \lambda), \; N'=(nodes(N'), \lambda')$ we say $N'$ is an extension of $N$, and write $N\subseteq N'$, if $nodes(N)\subseteq nodes(N')$ and the restriction of $\lambda'$ to $nodes(N)\times nodes(N)$ equals $\lambda$.  

   Henceforth, we may write $N$ for the name of the atomic labelling, its set of nodes and also the edge labelling function, distinguishing by context, e.g. $x\in N$ means $x$ is a node of $N$ and $N(x, y)$ denotes the atom labelling the edge $(x, y)$.    

The atomic network game $\mathsf{N} (\c A)$ is played by two players, $\forall$ and $\exists$, and has $\omega$ rounds.  In each round an atomic labelling $N_i$ is played ($i<\omega$).  In the initial round, $\forall$ picks any atom $a\in At(\c A)$ and $\exists$ must respond with an atomic labelling $N_0$ with nodes $x, y$ where $N_0(x, y)=a$ (the nodes $x, y$ may be distinct or not).   In round $i>0$ suppose $N_{i-1}$ was the last atomic labelling played. Then $\forall$ picks any two nodes $x, y\in N_{i-1}$ and any pair of atoms $a, b\in At(\c A)$ such that $(a, b, N_{i-1}(x, y))$ is not forbidden.    We denote his move $(x, y, a, b)$.   In response,  $\exists$ must play $N_i\supseteq N_{i-1}$ where there is $z\in N_{i}$ such that $N_i(x, z)=a,\; N_i(z, y)=b$.    If in any round the atomic labelling $N_i$ fails to be coherent (so fails to be an atomic network) then $\forall$ wins the play.  If $\forall$ does not win in any of the $\omega$ rounds then $\exists$ wins.

If the current atomic labelling is $N$, an $\forall$-move $(x, y, a, b)$ is \emph{trivial} if there is $z\in N$ with $N(x, z)=a, \; N(z, y)=b$.  $\exists$ can always respond to a trivial move by playing $N'=N$.  We will assume that $\forall$ avoids trivial moves. More details on atomic networks and games can be found in \cite[Chapter 11]{HH:book}. They are important here as, in certain circumstances, the chain of networks created during a play of $\mathsf{N}(\c A)$ corresponds to a representation of $\c A$. This is formalised by the following result.
\begin{proposition}
Let $\c A$ be a relation algebra with at most countably many atoms. Then $\c A$ is completely representable if and only if $\exists$ has a winning strategy in the atomic network game $\mathsf{N}(\c A)$.  
\end{proposition}
\begin{proof}
See \cite[theorem 11.7]{HH:book}.
\end{proof}
This proposition may be generalised to atomic  relation algebras with uncountably many atoms, but the atomic network game has to run transfinitely with as many rounds as atoms, taking limits of atomic networks at rounds indexed by limit ordinals.  Working out the details of this is given as \cite[exercise 11.4.3]{HH:book}, but we do not not need the general result here.

\section{The rainbow construction} \label{S:rainbow}
Given two sets $S, T$ we define an atomic relation algebra $\c B_{S, T}$ by defining its atom structure.   
The atoms are
\[ \set{1', \bb, \w, \y}\cup\set{ \g_i: i\in S}\cup\set{ \r_{j, j'}:    j, j'\in T}\]
The non-identity atoms are considered to be black, white, yellow, green or red.  
All atoms are self-converse, except $\r_{j, j'}^\conv = \r_{j', j}$.
Forbidden triples of atoms are Peircean transforms of
\begin{enumerate}
\renewcommand{\theenumi}{\Roman{enumi}}
\item $(1', a, b)$ where $a\neq b$\label{f:1}
\item $(\g_i, \g_{i'}, \g_{i''}), (\g_i, \g_{i'}, \w)$, any $i, i', i''\in S$\label{f:nog}
\item $(\y, \y, \y), (\y, \y, \bb)$\label{f:noy}
\item $(\r_{j_1, j_2}, \r_{j_2', j_3'} , \r_{j_1^*, j_3^*})$, unless $j_1=j_1^*, \; j_2=j_2',\;j_3'=j_3^*$\label{f:red}.
\item \label{f:pim}    $(\g_i, \g_i, \r_{j, j'}), \; (\g_i, \g_{i'}, \r_{j, j})$, any $i, i'\in S,\; j, j'\in T$
%$(\g_i, \g_{i'}, \r_{j, j'})$ unless $i\neq i'$ and $j\neq j'$.
\end{enumerate}
\begin{definition} The relation algebra $\c B_{S, T}$ is the complex algebra of this atom structure.  \end{definition}
This relation algebra  can be obtained from  the rainbow algebra of \cite[section 16.2]{HH:book} by regarding $S, T$ as `binary structures' in a relational  language with no  predicates except equality,   by deleting all white atoms $\w_X$ where $X\subseteq S$ has at most two elements, and by forbidding Peircean transforms of $(\g_i, \g_i, \r_{j, j})$ for $i\in S,\; j\in T$.

\begin{theorem} \label{thm:rainbow}
Let   $S, T$ be finite sets with $|S|\geq 2$.  Then the algebra $\c B_{S, T}$ is  representable if and only if $|S|\leq |T|$.
\end{theorem}
\begin{proof}
Suppose first that $\c B_{S, T}$ is   representable, and that $\theta$ is a  representation.  Since $\c B_{S, T}$ is finite, we know that $\theta$ is a complete representation.    
As $\w$ is an atom, and hence non-zero, there must be points $x, y$ in the base of the representation $\theta$ such that $(x, y)\in\w^\theta$ (equivalently, $(x, y)_\theta=\w$, recalling the notation from the end of section \ref{S:reps}). For each $i\in S$ the triple $(\g_{i},\y,\w)$ is not forbidden, so $\w \leq (\g_{i}\comp \y)$, and there must be a point $z_i$ with  $(x, z_{i})\in\g_{i}^\theta$ and $(z_{i}, y)\in\y^\theta$.  For distinct $i, i'\in S$ the atom $(z_i, z_{i'})_\theta$ must be red, by forbidden triples \eqref{f:1}--\eqref{f:noy}. Call this atom $\r(i, i')$, for some  element of $\set{\r_{j, j'}: j, j'\in T}$, and see the first part of figure \ref{F:injection}.
    For distinct $i_1, i_2, i_3\in S$ the first subscript of the red atom $\r(i_1, i_2)$ is the same as that of $\r(i_1, i_3)$ by forbidden triple \eqref{f:red}, see the second part of figure \ref{F:injection}.
Fixing $i\in S$, it follows that each $i'\in S$ corresponds to a point $\iota(i')\in T$ such that $(z_i, z_{i'})_\theta =  \r_{\iota(i), \iota(i')}$. As we have just noted, examining the second part of figure \ref{F:injection} we see that $\iota(i)$ is uniquely specified.  Moreover, by  the second forbidden triple of  \eqref{f:pim},  all indices $\iota(i')$ (for $i'\in S$) are distinct, hence the map $\iota:S \rightarrow T$ is an injection, and so $|S|\leq |T|$.

	\begin{figure}
\[ 
\xymatrix@!=2.5pc{
z_i\ar@{-}@/^/[rd]|-(.3){\y}\ar@{-->}[r]^{\r({i, i'})\in \set{\r_{j, j'}:j, j'\in T}} & z_{i'}\ar@{-}[d]^{\y}\\
x\ar@{-}[r]_{\w}\ar@{-}[u]^{\g_{i}}\ar@{-}@/_/[ur]|-(.4){\g_{i'}} & y
}\hspace{1in}
\xymatrix@!=2.5pc{
z_{i_1}\ar[r]^{\r_{j_1, j_2}} \ar[d]_{\r_{j_1', j_3'}} & z_{i_2} &\\
z_{i_3}\ar[ur]_{\r_{j_3^*, j_2^*}}& { \Rightarrow (j_1=j'_1)}
} 
\]
\caption{\label{F:injection}}
\end{figure}
	\medskip

For the converse, suppose $2\leq |S|\leq |T|$. We aim to show that $\c B_{S, T}$ is completely representable. As discussed above, it is sufficient to show that $\exists$ has a winning strategy in the complete representation game for $\c B_{S, T}$.  We  assume that $\forall$ never plays a trivial move, as $\exists$ can safely respond by leaving the current network unchanged. In particular we assume he never plays $(x, y, a, b)$ if $a= 1'$ or $b= 1'$.  

Suppose then that in a play of this game the current atomic labelling is $N$, and $\forall$ has just played the move $(x,y,a,b)$.  In her response, $\exists$ adds a new node $z$ to $N$ to create $N' = N\cup\{z\}$, and assigns labels to every node $w\in N\setminus\{x,y\}$ (see figure \ref{F:move}). Her strategy is as follows:

\begin{enumerate}[(a)]
\item
 If  $N(w, x)$ and $a$ are not both green, and $N(w, y), b$ are not both green, she sets $N'(w, z)=\w$. 
  \item If $N(w, x), a$ are both green but $N(w, y),b$ are not both yellow, or if $N(w, y), b$ are both green but $N(w, x), a$ are not both yellow, she sets $N'(w, z)=\bb$. 
  \item \label{case:rem} The remaining case is where, for some $i, i'\in S$, we have $N(x, w)=\g_{i},\;a=\g_{i'},\; N(w, y)=b=\y$ (or similar with $x, y$ swapped).  Since we are assuming that $\forall$ makes no trivial moves, we can assume $i\neq i'$.	
  		In this case $\exists$ can only use a red label for $N'(w,z)$, as the other colours are precluded by the forbidden triple rules. 
	To complete the definition of the strategy it remains to provide the two subscripts $j, j'\in T$ of the red atom she chooses for $N'(w, z)$.    	
	
	To help us here we introduce some notation, and make a few preliminary observations. 
Given an atomic  network $N$ for $\c B_{S, T}$ and nodes $x, y\in\nodes(N)$ let
\[R_N(x, y)=\set{z\in\nodes(N): N(x, z)\mbox{ is green and } N(y, z)=\y}.\]
Observe that $R_N(x, y)$ depends only on the green and yellow edge labels of $N$. A set of nodes of a network where every edge between distinct nodes has a  red label is called a \emph{red clique}, and it follows from the forbidden triple rules that $R_N(x, y)$ must be a red clique. 

By definition, part (c) of $\exists$'s strategy is relevant precisely when the new node $z$ is added to non-empty $R_N(x,y)$. The key idea is that every time a red clique of form $R_N(x,y)$ attains size two in the current network, $\exists$ will define an injection $h_{xy}:S\rightarrow T$ and this will be used to guide her choice of labels whenever a new $z$ is added to $R_N(x,y)$ in later rounds. This strategy will be well defined because, as $\exists$ never uses green or yellow labels, the only way the new node $z$ can be added to $R_N(x,y)$ is if $a$ is green and $b$ is yellow (or vice versa), in which case $x$ and $y$ are specified uniquely. Thus, given $z$, there is only one relevant $h_{xy}$.  

Let $N$ be the current atomic network in a play of the atomic representation game.
%, assume inductively that so far $\exists$ has not yet lost and $N$ is coherent.   Observe,  by \eqref{f:red}, that whenever $|R_N(x, y)|\geq 2$, there is a well-defined index map $\iota_{x,y}: R_N(x, y) \rightarrow n$, so that $N(w, w')=\r_{\iota_{x, y}(x), \iota_{x, y}(y)}$ for $w\neq w'\in R_N(x)$.  
Assume inductively:
if $|R_N(x, y)|\geq 2$ then there is an injection $h_{xy}:S\rightarrow T$ such that for all $w\neq w'\in R_N(x, y), \; i,  i'\in S$,  we have 
\begin{equation} \label{IH}(N(x, w)=\g_i \wedge N(x, w')=\g_{i'}) \Rightarrow (i\neq i'\wedge N(w, w') = \r_{h_{xy}(i), h_{xy}(i')}). \end{equation}
Initially all red cliques $R_{N_0}(x', y')$ are empty so the induction hypothesis holds.

There are three ways to extend the network $N$ to $N'$ so that a new red clique of form $R_{N'}(x,y)$ is created with $|R_{N'}(x,y)|=2$:
\begin{enumerate}[(i)] 
\item The first is if $|R_N(x,y)|=1$, say $R_N(x, y)=\set w$ where $w\in  N$,  and $\forall$ plays $(x,y,\g_{i'},\y)$ for some  $i'\in S$.  
In this case, $R_{N'}(x, y)=\set{w, z}$, with $N(x, w)=\g_{i}$ for some $i\neq i'\in S$.  Here $\exists$ defines $h_{xy}$ 
to be any injection $h_{xy}:S\to T$. Such an injection must exist as we are assuming $|S|\leq |T|$.  Then she sets $N'(w, z)=\r_{h_{xy}(i), h_{xy}(i')}$, in accordance with \eqref{IH}.

\item The second way such a red clique of size  two can be created is where $w\in N$, \/ $z$ is the new node, and $\forall$'s move is $(x,y,\y,\y)$, with $N(w,x)= \g_i$ and $N(w,y) = \g_{i'}$ for some $i, i'\in S$ (see figure \ref{F:clique}). Since $N$ is an atomic network and $\forall$'s move is legal, $N(x, y)$ must be red and $i\neq i'$ (using the first forbidden triple of \eqref{f:pim}). This creates the red clique $R_{N'}(w,z)=\set{x, y}$ ($\set{x, y}$ was previously a red clique, but not of the form $R_N(x', y')$).   Note that $x$ and $y$ can be the only members of $R_{N'}(w,z)$ as $\exists$ never uses yellow labels, so no other network node can have a yellow edge to $z$.  Here the label $N(x,y)$ was defined in $N$ previously, and must be $\r_{j, j'}$ for some $j\neq j'\in T$, as otherwise $\forall$'s move would be illegal. In this case $\exists$ defines $h_{wz}$ to be any injection $S\rightarrow T$ with $h_{wz}(i) = j$ and $h_{wz}(i') = j'$.

\item The third way is similar to the second, except $\forall$ plays $(x,y,\g_i,\g_{i'})$, and $N(w,x)=N(w,y)=\y$. This is similar to the previous case, and for $\forall$'s move to be legal we must have $N(x,y) = \r_{j, j'}$ for some $j\neq j'\in T$. Again $\exists$ defines $h_{wz}$ to be any injection with $h_{wz}(i) = j$ and $h_{wz}(i') = j'$.
\end{enumerate}

The only case where $\exists$ chooses any red labels is when  $\forall$'s move is $(x, y, \g_i, \y)$ or $(x, y, \y, \g_i)$, for some $i\in S$. So the only way a red clique of form $R_{N'}(x,y)$ with size $k>2$ can occur is where $|R_N(x, y)|=k-1$ and $\forall$'s move for the round is $(x, y, \g_i, \y)$ for some $i\in S$.   Since we are assuming that no trivial moves are played, there is no $w\in R_N(x, y)$ with $N(x, w)=\g_i$.  In this case $\exists$ adds a single new node $z$ to the network and includes it in $R_{N'}(x, y)$. Then, according the $\exists$'s strategy, edges $(z, w)$ have red labels if $w\in R_N(x, y)$ but either white or black labels if $w\in N\setminus R_N(x, y)$.  Note that $z\in R_{N'}(x', y')$ if and only if $x=x'$ and $y=y'$.  Since $k-1\geq 2$ we know inductively that there is an injection $h_{xy}:S\rightarrow T$, satisfying \eqref{IH}.  For $w\in R_N(x, y)$ where $N(x, w)=\g_{i'}$ say, she lets $N(w, z) = \r_{h_{xy}(i'), h_{xy}(i)}$ (see figure \ref{F:label}). This maintains \eqref{IH} and completes the definition of $\exists$'s strategy.
	\end{enumerate}
	
	\begin{figure}
\[ 
\xymatrix@!=2.5pc{
& w\ar[d]_{c_w} \\
& z\ar[dr]_b \\
x\ar[ur]_a\ar@/^1pc/[uur]^{N(x,w)}\ar[rr]_{N(x,y)} & & y\ar@/_1pc/[uul]_{N(y,w)}
}
\]
\caption{\label{F:move}For each $w\in N\setminus\{x,y\}$, \/ $\exists$ must assign the label $c_w=N'(w,z)$.}
\end{figure}
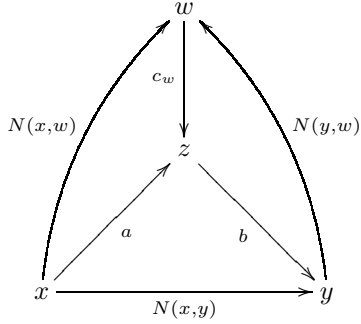

	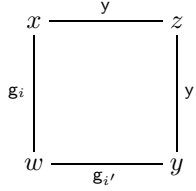
\begin{figure}
\[ 
\xymatrix@!=2.5pc{
x\ar@{-}[r]^{\y} & z\ar@{-}[d]^{\y} \\
w\ar@{-}[u]^{\g_i}\ar@{-}[r]_{\g_{i'}} & y
}
\]
\caption{\label{F:clique}Creating a red clique of size 2.}
\end{figure}

%Returning to $\exists$'s strategy, we are assuming that $z$ is being added to $R_N(x,y)$ for some $x,y\in N$, and that $|R_N(x,y)|\geq 1$, so there is some $w\neq z\in R_N(x,y)$ for her to label $N'(w,z)$. Moreover, by assumption $N(x,w) = \g_i$, and $N(x,z)=\g_{i'}$. Here she sets $N'(w,z) = \r_{h_{xy}(i),h_{xy}(i')}$ (see figure \ref{F:label}), where $h_{xy}$ was either defined already for the pair $(x,y)$, or, if the addition of $z$ extends the size of $R_N(x,y)$ to 2, defined now, as discussed above in (i). As mentioned previously, this choice is well defined as, if $z \in R_N(x',y')$, then necessarily $x=x'$ and $y=y'$.

	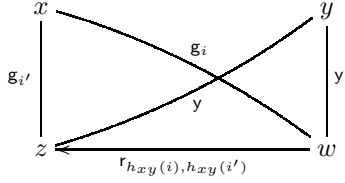
\begin{figure}
\[ 
\xymatrix@!=2.5pc{
x\ar@{-}[d]_{\g_{i'}}\ar@{-}@/^/[drr]^{\g_i} & & y\ar@{-}[d]^{\y}\ar@{-}@/^/[dll]^{\y} \\
z & & w\ar[ll]^{\r_{h_{xy}(i),h_{xy}(i')}}
}
\]
\caption{\label{F:label}The label $N'(w,z)$ in case (c).}
\end{figure}
	
  We must show that this strategy is a good one by proving that $N'$ is an atomic network, i.e. that the labelling is coherent. We must check that the labeling of each triangle $(w, w^*, z)$ for $w\neq w^*\in N$ is not forbidden.  If $\set{w, w^*}=\set{x, y}$ then the triangle is consistent (else $\forall$'s move would be illegal), so without loss of generality we assume that $w\notin \{x,y\}$, and so the label $N'(w,z)$ was just assigned by $\exists$ in this round.  Hence $N'(w, z)$ must be either white, black or red, since she only chooses these colours.  If $N'(w, z)=\w$, then the only possibility that the triangle $(w, w^*, z)$ could be forbidden comes from \eqref{f:nog}, but this requires that $N'(w^*,z)$ and $N'(w,w^*)$ be green, and thus that $w^*\in\{x,y\}$. But then the conditions of (a) would not have been met, so $N'(w,z)$ could not be $\w$ after all.  Similarly, if $N'(w, z)=\bb$ then the possibility of violating \eqref{f:noy} is ruled out by  case (b)  conditions.    
	
	In the remaining case, $N'(w, z)$ is red, and the only forbidden triples involving red atoms are  \eqref{f:red} and \eqref{f:pim}.  Since $w\notin\{x,y\}$ by assumption, the label $N'(w, z)$ must have been assigned according to part (c) of $\exists$'s strategy, so we assume that $N(x, w)=\g_i,\;a=\g_{i'}$, and $N(y,w)=b =\y$ (the case where $x, y$ are swapped is symmetric).      A triangle $(w, w^*, z)$ could only violate forbidden triple \eqref{f:red} if all three edges were red. For this to happen we must have $N'(x,w^*) = \g_{i^*}$ for some $i^*\in S$, and $N'(y,w^*) = \y$, so $\set{w, w^*}\subseteq R_N(x, y)$.  By our induction hypothesis, there is an injection $h_{xy}:S\rightarrow T$ satisfying \eqref{IH} in $N$.  According to part (c) of her strategy, $\exists$ set $N'(w, z)=\r_{h_{xy}(i), h_{xy}(i')}$ and $N'(w^*, z)=\r_{h_{xy}(i^*), h_{xy}(i')}$.   It follows that $(w, w^*, z)$ does not violate \eqref{f:red} (see figure \ref{F:triangle}).

%	, so we need only check that the label $N(w,w^*)$ must be $\r_{h_{xy}(i),h_{xy}(i'')}$, and thus that the triangle does not violate rule \eqref{f:red}. If the label $N(w,w^*)$ was added after the definition of $h_{xy}$, then this must hold, as in this case at least one of $w$ and $w^*$ was added to the clique after $h_{xy}$ was defined, and thus $N(w,w^*)$ was defined using $h_{xy}$. Alternatively, if the label $N(w,w^*)$ existed in the network before $R_N(x,y)$ attained size 2, and thus before $h_{xy}$ was defined, then the clique must have been created according to either (ii) or (iii) as described above. In both these cases $h_{xy}$ is \emph{defined} so that $N(w,w^*)= \r_{h_{xy}(i),h_{xy}(i')}$. 

	\begin{figure}
\[ 
\xymatrix@!=2.5pc{
& x\ar@{-}[dl]_{\g_i}\ar@{-}@/^/[dr]^{\g_{i^*}}\ar@{-}[rr]^{\g_{i'}} & & z \\
w\ar@/_/[rrru]^(.42){\r_{h_{xy}(i),h_{xy}(i')}}\ar[rr]_{\r_{h_{xy}(i),h_{xy}(i^*)}} & & w^*\ar[ur]_{\r_{h_{xy}(i^*),h_{xy}(i')}}
}
\]
\caption{\label{F:triangle} The red triangle involving $w,w^*$ and $z$.}
\end{figure}
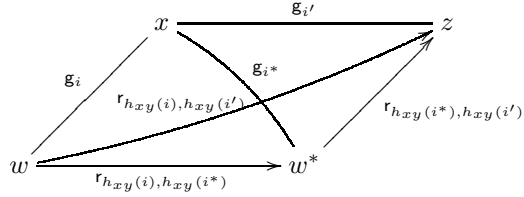

For rule\eqref{f:pim},  the only possible  green-green-red triangle incident with $z$ and $w$ is $(x, w, z)$ and the edge labels  $(\g_i, \g_{i'},\r_{h_{xy}(i), h_{xy}(i')})$ do not violate \eqref{f:pim}, since $h$ is injective.  Hence the labelling is coherent, and $N'$ is an atomic network.
\end{proof}
This theorem may be generalised: if $S, T$ are arbitrary sets where $T$ has at least two elements, then $\c B_{S, T}$ is completely representable if and only if the cardinality of $S$ is no more than that of $T$.  The left to right implication is proved as above.  For the right to left implication we  use a transfinitely long atomic network game with as many rounds as $|S|+\omega$, as mentioned above.   The relation algebras of interest here are finite, so we leave out a formal proof of the more general result.

The class of atom structures of representable relation algebras is known to be elementary \cite{V97:atom}.   We can apply the result above to say something about classes of atom structures intermediate between this class and that of the atom structures of completely representable relation algebras. 

\begin{corollary}\label{C:atoms}
   If $\c K$ is a class of relation algebra atom structures including all atom structures of completely representable relation algebras, and contained in the class of atom structures of representable atomic relation algebras, then $\c K$ cannot be defined by any theory in the language of RA atom structures using only finitely many  atom valued variables.
\end{corollary}
\begin{proof} 

Given finite sets $S, T$  where $2\leq |S|<|T|$ consider the algebras $\c B_{S, S}$ and $\c B_{T, S}$.  By theorem~\ref{thm:rainbow}, the former is completely representable while the latter is not, and, since they are finite, representability is the same as complete representability. Thus the atom structure of $\c B_{S, S}$ is in $\c K$ and that of $\c B_{T, S}$ is not. 

However, consider the well-known Ehrenfeucht-Fra\"iss\'e game for testing elementary equivalence played over a pair of relational structures using $c$ pebbles. A full description of this game can be found in e.g. \cite[chapter 6]{Im99}, but a brief summary is that, in each round, the player $\forall$ associates a pebble with an element of his choice from one of the structures (again of his choice), and then the player $\exists$ associates the same pebble with an element of her choice from the other structure. This association via pebbles induces a mapping between the two structures, and $\forall$ wins if at any point it fails to be a well-defined isomorphism. The key result is that $\exists$ has a strategy guaranteeing her survival in this game for at least $n$ rounds if and only if the two structures are equivalent with respect to all formulas of quantifier depth at most $n$ using at most $c$ free variables  (see e.g. \cite[theorem 6.10]{Im99}).     

In our case, $\exists$ has a winning strategy in the $|S|$-pebble, $\omega$-round \ef-game over the atom structures of $\c B_{S, S}$ and $\c B_{T, S}$, which we describe now. Whenever $\forall$ picks a non-green atom, $\exists$ picks the corresponding non-green atom in the other atom structure.  If $\forall$ places a pebble where another pebble is already placed, then $\exists$ covers the corresponding pebble in the other structure. If $\forall$ picks a green atom not already in play, then $\exists$ picks any green atom in the other atom structure not currently selected. There are always enough green atoms for this, because the game only uses $|S|$ pebbles. This strategy is a winning one because a triple $(\g_i, \g_{i'},\r_{j,j'})$ is forbidden iff $i=i'$ or $j=j'$, in either atom structure. It follows that the two atom structures agree on all $|S|$-variable formulas .

\end{proof}
   
The corollary shows that the atom structures of $\c B_{S, S}$ and $\c B_{T, S}$ cannot be distinguished in the language of atom structures restricted to $|S|$ atomic variables. 
If we use formulas with variables that range over arbitrary elements of a relation algebra, much more  can be expressed.  
Consider, for example, the formula $\phi_k(x)$ with  variables $x, y$ of which only $x$ appears free which we will define shortly.  It is intended to express that $x$ is above  at least $k$ atoms, in an atomic relation algebra.  So $\phi_1(x)$ is $\neg(x=0)$.  Recursively, suppose $\phi_k(x)$ and $\phi_k(y)$ have been defined (the variables $x$ and $y$ are swapped throughout in the latter formula),  and  suppose the formula holds exactly when the free variable denotes an element  above at least $k$ atoms.  Let $\phi_{k+1}(x)$ be the formula $\exists y(y<x\wedge\phi_k(y))$, where $<$ is the usual shorthand for the formula defining strict  order in boolean algebras.
If $k$ is the number of atoms in $\c B_{T, S}$ then $\exists x \phi_k(x)$ is true in  $\c B_{T, S}$ but not in $\c B_{S, S}$.  As well as defining the finite cardinality of an algebra, two variable formulas can express many other properties. Indeed, it is conceivable that any pair of non-isomorphic  finite relation algebras can be distinguished by a two variable formula; this remains an open problem.  

The proof of corollary \ref{C:atoms} is a kind of warm up for the proof of theorem \ref{thm:main}, which involves formulas with variables ranging over arbitrary relation algebra elements, and  occupies most of the next section.

\section{Axiomatisations of RRA}\label{S:axioms}
We will use the following minor variation of the classic Ehrenfeucht-Fra\"iss\'e game used in the proof of corollary \ref{C:atoms}. Given $n<\omega$ and two relation algebras $\c A, \c B$   we define the  $n$-round equivalence game $\Gamma_n(\c A, \c B, p)$, where $p$ is a sequence of pairs from $\c A\times\c B$.  The idea of the game is to test whether $p$ could define a function extending to an isomorphism from $\c A$ to $\c B$.  

A \emph{play }$(p_0,\ldots, p_n)$ of the game $\Gamma_n(\c A, \c B, p)$ consists of a sequence of sequences of pairs from $\c A \times \c B$, with $p_0 = p$, and where each $p_{i+1}$ is obtained by appending a pair $(a, b)\in\c A \times \c B$ to $p_i$.  These sequences are called \emph{positions}, and $p = p_0$ is the \emph{starting position}.

  If $n = 0$ then nothing happens, and the result of the game will depend only on the starting position $p$. For $n\geq 1$, in round $i<n$, the current position is $p_i$ and $\forall$ chooses either $a\in\c A$ or $b\in\c B$ as he prefers, and $\exists$ chooses the other element. The position $p_{i+1}$  is obtained by appending the pair $(a, b)$ to  $p_{i}$.
The game $\Gamma_n(\c A, \c B, \emptyset)$ that starts from the empty position is denoted $\Gamma_n(\c A, \c B)$. 

Given a position $q = ((a_0,b_0),\ldots,(a_k,b_k))$ arising during a play of $\Gamma_n(\c A, \c B, p)$, define the sequences $\bar{a} = (a_0,\ldots,a_k)$ and $\bar{b} = (b_0,\ldots,b_k)$.
Let $\c A_{\bar a}$ and $\c B_{\bar b}$ 
denote the subalgebras of $\c A$ and $\c B$ generated by  $\bar a$ and $\bar b$, respectively. We want to use the sequence $q$ to define a map $\lr{q}$ from $\c A_{\bar a}$ to $\c B_{\bar b}$. To do this note that elements of $\c A_{\bar a}$ correspond to terms constructed from elements of $\bar a$ and relation algebra constants using relation algebra operations, and similar for elements of $\c B_{\bar b}$. Given such a term $t$ in $\c A_{\bar a}$ we want to define $\lr{q}(t)$ to be the term in $\c B_{\bar a}$ obtained by replacing each $a_i$ with $b_i$ and preserving relation algebra constants. If $\lr{q}$ fails to be an isomorphism, or fails to be well defined at all, then $q$ is a winning position for $\forall$.  

If any position in the play is a winning position for $\forall$ then he wins. If none of the positions are winning positions for $\forall$ then $\exists$ wins.  Since a winning position for $\forall$ remains a winning position for $\forall$ after further play, it follows that  the winner is determined by the final position $p_{n}$. The result is a win for $\exists$ if and only if $\lr{p_n}(t)$ is a well-defined isomorphism. Note that if during the course of the game $\forall$ chooses an element he already chose in a previous round, $\exists$ can respond by choosing the element she chose in that round, and the outcome of the game is unchanged. So we assume without loss of generality that $\forall$ never plays these redundant moves.

The value of these Ehrenfeucht-Fra\"iss\'e games for relation algebras is given by the following definition and lemma.

	\begin{definition}
	Let $n<\omega$, let  $\c A$ and $\c B$ be relation algebras,  and let $\bar a=(a_0, \ldots, a_{k-1})$ and $\bar b=(b_0, \ldots, b_{k-1})$ be tuples from $\c A, \c B$ respectively, of the same length $k$.
	 We write 
	\[(\c A, \bar {a}) \equiv_n (\c B, \bar b)\]
	if whenever $\phi$ is a first-order formula of quantifier depth at most $n$, with free variables from $\set{x_i:i<k}$, in the language of relation algebras,  we have 
	\[\c A, \bar a \models \phi \iff \c B, \bar b \models \phi.\]
	%Here, for example, $\c A, \alpha \models \phi$ means that $\c A \models \phi$ if all variables $x_i$ occurring free in $\phi$ are assigned to $\alpha(i)$ in $\c A$.
	When  $\bar a$ and $\bar b$ are empty we just write $\c A \equiv_n \c B$.
	\end{definition}

\begin{lemma}\label{L:EF}
Let  $n < \omega$, let $\bar a=(a_0, \ldots, a_{k-1}), \;\bar b=(b_0, \ldots, b_{k-1})$ be tuples from $\c A, \c B$ respectively and let  $p=((a_0, b_0),\ldots,(a_{k-1},b_{k-1}))$.  Then  $\exists$ having a winning strategy in $\Gamma_n(\c A,  \c B, p)$  implies $\c A, \bar{a} \equiv_n \c B, \bar{b}$. 
\end{lemma}
\begin{proof}
This is half the well known result for relational signatures used in the proof of corollary \ref{C:atoms} (see e.g. \cite[theorem 6.10]{Im99}). Having functions in the signature blocks the proof of the opposite implication. We induct on $n$. For the base case, if $\exists$ wins  $\Gamma_0(\c A, \c B, p)$ then $p$ induces an isomorphism from $\c A_{\bar a}$ to $\c B_{\bar b}$,  hence $(\c A, \bar{a})\equiv_0(\c B, \bar{b})$.

For the inductive step, suppose $\exists$ has a winning strategy in $\Gamma_{n+1}(\c A, \c B, p)$, and let $\phi$ be a formula of quantifier depth $n+1$. Note that, for all formulas $\phi_1$ and $\phi_2$,  if $\c A$ and $\c B$ disagree about $\phi_1\wedge \phi_2$ they must also disagree about either $\phi_1$ or $\phi_2$, and if they disagree about $\neg \phi_1$ they must also disagree about $\phi_1$.   So we can assume without loss of generality that $\phi=\exists x_k \psi$, where $\psi$ is a formula of quantifier depth $n$.      If $\c A, \bar{ a}\models \exists x_k\psi$ then there is $a_k\in\c A$ such that   $\c A, \bar{ a},a_k\models \psi$.  If $\forall$ plays $a_k$ in the game, then since $\exists$ has a winning strategy there is  $b_k\in\c B$ where $\exists$ has a winning strategy in $\Gamma_n(\c A,  \c B, p')$, where $p'$ is  $p$ with  $(a_k, b_k)$ appended.    Inductively, $\c B,\bar b, b_k \models\psi$, hence $\c B, \bar{b}\models\exists x_k \psi$.  Since the argument is symmetric, it follows that $(\c A, \bar a)$ agrees with $(\c B,\bar b)$ on all  formulas  $\exists x_k\psi$ where $\psi$ has quantifier depth at most $n$, hence they agree on all  formulas of quantifier depth at most $n+1$.      By induction, the lemma holds for all  $n<\omega$. 
% For the case $n=\omega$ then a winning strategy for $\exists$ in $\Gamma_\omega(\c A, \c B, p)$ entails a winning strategy in all games  $\Gamma_n(\c A, \c B, p) $ so $\c A, \bar a\equiv_n\c B, \bar b$ for all finite $n$.  Hence $(\c A, \bar{a})\equiv_\omega(\c B, \bar{b})$, as required.
\end{proof}

 We now define a colouring game played by $\forall$ and $\exists$ over a pair of sets $(T, T')$. Colours are used to colour \emph{subsets} of $T$ and $T'$, rather than the individual elements used in the pebble games of corollary \ref{C:atoms} and lemma \ref{L:EF}.
Let $n< \omega$ and let $T, T'$ be sets. We define the $n$-round colouring game $\G_n(T, T')$. If $n=0$ then the game ends immediately with neither player making a move. For $n\geq 1$, play of the  game  is a sequence $((T_0, T'_0), \ldots, (T_{n-1}, T'_{n-1}))$ where $T_i\subseteq T,\; T'_i\subseteq T'$, for $i<n$. The positions in a play are its initial segments. The initial position $p_{0}$ is the empty sequence. For $n > 0$, at the start of round $i<n$ the position is denoted $p_i$. Then  $\forall$ chooses either a subset $T_i\subseteq T$ or a subset $T'_i\subseteq T'$, as he prefers, and $\exists$ chooses the other subset. The position is then updated to $p_{i+1}$, which is  $p_i$ with the pair $(T_i,T'_i)$ appended. Note that the final position $p_n$ is the full play $((T_0, T'_0), \ldots, (T_{n-1}, T'_{n-1}))$.     

The numbers $i<n$ denote colours, and we can think of $\forall$ and $\exists$ taking it in turns to paint subsets of $T$ and $T'$ with different colours. Continuing the painting analogy, a \emph{palette}
$\pi$ is a subset of $\set{i:i<n}$, in other words, a choice of colours. When $n=0$ the only palette is $\emptyset$. For $0<i<n$, at position $p_{i} =((T_0, T'_0),\ldots,(T_{i-1},T'_{i-1}))$ we may interpret a palette $\pi$  in $T$ and in $T'$, by

\begin{align*}\pi^{p_{i}}_T=\{x\in T:\forall j<i(x\in T_j \iff j\in \pi)\}\\ \pi^{p_{i}}_{T'}=\{x\in T':\forall j<i(x\in T'_j \iff j\in \pi)\}.
\end{align*}

%\begin{equation}\label{def:pal}\pi^{p_{i}}_T=\bigcap_{j< i,\; j\in\pi}T_j\setminus \bigcup_{j< i,\; j\not\in\pi}T_j\hspace{1in} \pi^{p_{i}}_{T'}=\bigcap_{j< i,\; j\in\pi}T'_i\setminus \bigcup_{j< i,\; j\not\in\pi}T'_j.\end{equation} 

Intuitively, $\pi^{p_{i}}_T$ is the set of vertices of $T$ with exactly the combination of colours defined by $\pi$ at position $p_{i}$. We define $\pi^{p_0}_T= T$ and $\pi^{p_0}_{T'}= T'$ for all palettes $\pi$. 
Note also that for every non-empty position $p$, the set of vertices of $T$ is the disjoint union of the sets $\pi^{p}_T$, as $\pi$ ranges over palettes, and similar for $T'$.
A position $p$ is a win for $\forall$ if  there is a palette $\pi\subseteq\set{i:i<n}$ where
\begin{align}  \label{eq:win} |\pi^{p}_T|\neq|\pi^{p}_{T'}| \mbox{ and either }|\pi^{p}_T|< 2 \mbox{ or } |\pi^{p}_{T'}|< 2.
\end{align}
If for any $i\leq n$ position $p_i$ is a win for $\forall$ then $\forall$ wins the play, but if no position in the play is a win for $\forall$ then $\exists$ wins.  Since the play is finite, $\forall$ wins if and only if the final position $p_n$ is a win for him.  

We call this game the \emph{Seurat game for sets}  played over $T$ and $T'$. This is a variant  of the Seurat games for digraphs and general binary structures defined in \cite{EgrHirgames} and \cite{EgrHirNote}, respectively, but with no restriction to the number of colours and no winning condition for $\forall$ relating to edges or binary relations. The Seurat game for sets over $T, T'$ can be thought of as a special case of a Seurat game for digraphs by thinking of $T$ and $T'$ as being complete graphs.   %This game we use here may be obtained from the Seurat game for digraphs with $\omega$ colours and no re-use of colours,  by treating the sets $S$ and $T$ as complete graphs.

\begin{lemma}\label{L:K}
Let $n < \omega$, and let $T$ and ${T'}$ be sets of size at least $2^{n+1}$. Then $\exists$ has a winning strategy in $\G_{n}(T, T')$.
\end{lemma}
\begin{proof}
If $|T|,|T'|\geq 2$ then $\exists$ wins $\G_{0}(T, T')$ as $\pi_T^{p_0} = T$ and $\pi_{T'}^{p_0}=T'$, so assume $n>0$.
Suppose $\exists$ plays according to the following principle:  If the position is $p_{r}$, for every palette $\pi\subseteq\set{i:i<n}$, she ensures that
\begin{equation*}\tag{$\dagger_r$}
\label{t:dag} 
(|\pi^{p_{r}}_{T}|< 2^{n+1-r}\vee |\pi^{p_{r}}_{T'}|<  2^{n+1-r}) \rightarrow |\pi^{p_r}_{T}|=|\pi^{p_r}_{T'}|.
\end{equation*}
If $\exists$ can maintain $(\dagger_r)$ while $r\leq n$ she will survive all rounds.  At the end of the final round $n$, $(\dagger_{n})$ ensures that \eqref{eq:win} holds, but note that if the game were  to continue for another round then she might lose, because $(\dagger_{n+1})$ does not insure her  against violating \eqref{eq:win}.

We now prove by induction that $\exists$ can indeed always play so as to ensure \eqref{t:dag}  holds up to and including $r=n$. The base case is $r=0$. The position here is $p_0$, and by definition $\pi^{p_0}_T= T$ and $\pi^{p_0}_{T'}=T'$ for all palettes $\pi$, and so $(\dagger_{0})$ holds as by assumption $|T|,|T'|\geq 2^{n+1}$.  

Suppose now that $0\leq r< n$, that the position is $p_r$, and that $(\dagger_{r})$ holds.  Suppose without loss of generality that $\forall$ picks $T_r\subseteq T$ (the case where he chooses a subset of $T'$ is similar).  For $\exists$'s response $T'_r\subseteq T'$, she will pick disjoint subsets $T'_\pi\subseteq \pi^{p_{r}}_{T'}$ (for each palette $\pi$) and then she will define $T'_r=\bigcup_{\pi\subseteq\set{i:i<n}} T'_\pi$.  How she chooses each $T'_\pi$ is explained next. Note that if $\pi$ is a palette, then $\pi^{p_{r+1}}_T$ will either be $\pi^{p_{r}}_T\cap T_{r}$ or $\pi^{p_{r}}_T\setminus T_{r}$, depending on whether or not $r\in \pi$, and similar for $\pi^{p_{r+1}}_{T'}$. So maintaining $(\dagger_{r+1})$ comes down to ensuring these sets have the right cardinalities.  
\begin{itemize}
\item  If $|\pi^{p_{r}}_T\cap T_{r}|< 2^{(n+1)-(r+1)} = 2^{n-r}$ and $|\pi^{p_{r}}_T\setminus T_{r}|< 2^{n-r}$, then $|\pi^{p_{r}}_T|< 2^{n+1-r}$, so by the inductive assumption $(\dagger_{r})$ we have $|\pi^{p_{r}}_T|=|\pi^{p_{r}}_{T'}|$.  Here she lets $T'_\pi$ be any subset of $\pi^{p_{r}}_{T'}$ of size $|\pi^{p_{r}}_T\cap T_{r}|$. It follows that $\pi^{p_{r}}_{T'}\setminus T'_\pi$ has the same size as $\pi^{p_{r}}_T\setminus T_r$.
 
\item  If $|\pi^{p_{r}}_T\cap T_r|< 2^{n-r}$ but $|\pi^{p_{r}}_T\setminus T_r|\geq 2^{n-r}$ then she lets $T'_\pi$ be any subset of $\pi^{p_{r}}_{T'}$ of the same size as $\pi^{p_{r}}_T\cap T_r$. It follows that $\pi^{p_{r}}_{T'}\setminus T'_\pi$ will have size at least $2^{n-r}$, because otherwise $|\pi^{p_r}_{T'}|<2^{n+1-r}$, and so by $(\dagger_{r})$ we would have $|\pi^{p_r}_{T}| = |\pi^{p_r}_{T'}|$, and thus $2^{n-r}\leq |\pi^{p_{r}}_T\setminus T_r| = |\pi^{p_{r}}_{T'}\setminus T'_\pi|< 2^{n-r}$, which would be a contradiction.  
 
\item Similarly, if $|\pi^{p_{r}}_T\setminus T_r|< 2^{n-r}$ but $|\pi^{p_{r}}_T\cap T_r|\geq 2^{n-r}$ she chooses $T'_\pi$ so that $\pi^{p_{r}}_{T'}\setminus T'_\pi$ has the same size as $\pi^{p_{r}}_T\setminus T_r$.

\item  Finally, if both $\pi^{p_{r}}_T\cap T_r$ and $\pi^{p_{r}}_T\setminus T_r$ have size at least $2^{n-r}$ then $|\pi^{p_{r}}_T|\geq 2^{n+1-r}$, so inductively $|\pi^{p_{r}}_{T'}|\geq 2^{n+1-r}$.  She lets $T'_\pi\subseteq \pi^{p_{r}}_{T'}$ be any subset of size $2^{n-r}$, and so $|\pi^{p_{r}}_{T'}\setminus T'_\pi|\geq 2^{n-r}$.
\end{itemize}
So
\begin{align*} 
\pi^{p_{r+1}}_T&=\begin{cases} \pi^{p_{r}}_T\setminus T_r\mbox{ if }r\not\in\pi\\ \pi^{p_{r}}\cap T_r\mbox{ if }r\in\pi \end{cases}\\
\pi^{p_{r+1}}_{T'}&= \begin{cases} \pi^{p_{r}}_{T'}\setminus T'_r = \pi^{p_{r}}_{T'}\setminus T'_\pi \mbox{ if } r\not\in\pi \\ \pi^{p_{r}}\cap T'_r\mbox{ if }r\in\pi \end{cases}
\end{align*}

By the definition by cases given above, the cardinalities of these sets agree when necessary. Thus ($\dagger_{r+1}$) is established for the new position  $p_{r+1}$.
\end{proof}

Let $T, T', S$ be finite sets.   We are interested in the game $\Gamma_n(\c B_{T, S}, \c B_{T', S})$. Recall that $\c B_{T, S}$ and $\c B_{T', S}$ are complex algebras generated by the `rainbow' atom structures defined in section \ref{S:rainbow}. Note that $\c B_{T, S}$ and $\c B_{T', S}$ are both finite (and thus atomic), and differ only with respect to their sets of green atoms (we identify the non-green atoms between each algebra in the obvious way). 
Given an element $x\in \c B_{T, S }$, we define the subset of elements of $T$ indexing green atoms below $x$ in the standard boolean ordering of $\c B_{T, S}$ to be $T_x$. 
%Similarly, given $x\in \c B_{m, l}$, we define the  subset of  elements of $m$ indexing green atoms below $y$ to be $Y_x$. 
 In addition, given an element $x\in \c B_{T, S}$ we define the \emph{green part} of $x$ to be the join of the set of green atoms below $x$, and the non-green part to be the join of all the other atoms below $x$. We say $x$ is \emph{green}, if it is a sum of green atoms. We make analogous definitions for $y\in \c B_{T', S}$.

\begin{proposition}\label{P:K-equiv}
Let $T, T', S$ be finite sets, and let $n<\omega$. If $\exists$ has a winning strategy in $\G_{n}(T, T')$, then $\exists$ has a winning strategy in $\Gamma_{n}(\c B_{T, S}, \c B_{T', S})$
\end{proposition}
\begin{proof}
$\exists$'s strategy is to simulate a corresponding play of $\G_{n}(T, T')$ in which she uses her winning strategy,  and to maintain a correspondence between the plays of the games.  So, if $\forall$ selects $x\in\c B_{T, S}$, then she selects $y\in \c B_{T', S}$ whose non-green part is identical to that of $x$ and whose green part is defined by her response to the $\forall$-move $T_x$ in the play of $\G_{n}(T, T')$, and similar when he picks an element of $\c B_{T', S}$.

We assume that $n\geq 1$, as the argument for the $n=0$ case is essentially the same, but simpler as all that is considered are the empty starting positions in both games. So let $p_{n}=((a_0, b_0),\ldots,(a_{n-1},b_{n-1}))$ be the position at the end of the final round in a play of the game $\Gamma_{n}(\c B_{T, S}, \c B_{T', S})$ in which $\exists$ uses this strategy, and let $((T_0, T'_0),\ldots,(T_{n-1},T'_{n-1}))$ be the corresponding  position at the end of the final round in the corresponding play of $\G_{n}(T, T')$. Note that $p_{n}$ is not a winning position for $\forall$ if and only if $p_i$ is not a winning position for him for all $i\leq n$, so it is sufficient for us to prove that $\forall$ does not win at the end of round $n$. 

For each $j<n$ we have $T_j = T_{a_j}\subseteq T$, where $T_{a_j}$ is the set of elements of $T$ indexing the green atoms below $a_j$, and $T'_j=T'_{b_j}\subseteq T'$ similarly.  From the assumption that $\exists$ is playing according to a winning strategy in $\G_{n}(T, T')$  we have
\begin{equation}\label{t:dag2}
 \text{either $|\pi^{p_{n}}_T|=|\pi^{p_{n}}_{T'}|$ or $|\pi^{p_{n}}_T|, |\pi^{p_{n}}_{T'}|\geq2$, for all palettes $\pi$},\end{equation}
and we have to prove that $p_{n}$ is not a losing position in the game $\Gamma_{n}(\c B_{T, S}, \c B_{T', S})$. If we define $\bar a = (a_0, \ldots, a_{n-1})$ and $\bar b = (b_0,\ldots, b_{n-1})$, it is sufficient to show that  the map $\bar a \mapsto \bar b$ induces a relation algebra isomorphism from $(\c B_{T, S})_{\bar  a}$  to $(\c B_{T', S})_{\bar b}$ (recall that these are the subalgebras generated by $\bar a$ and $\bar b$ respectively).

   Consider the boolean subalgebra $\c B$ of $\c B_{T, S}$ generated (using boolean operators) by $\set{a_0, \ldots, a_{n-1}}$ and all the non-green atoms.   Given a palette $\pi$, let $\pi^{\c B}=\sum_{i\in\pi^{p_{n}}_T}\g_i\in \c B$.  If $\pi^{p_{n}}_T=\emptyset$ then  the sum is empty and $\pi^{\c B}=0$, else  the sum is non-empty and $\pi^{\c B}$  is an atom of $\c B$ (though not usually an atom of $\c B_{T,S}$). To see that non-zero $\pi^{\c B}$ is an atom note that it has zero intersection with every non-green atom, and its intersection with an element $a_i$ is either itself, when $i\in \pi$, or zero otherwise.  A little thought reveals that all green atoms of $\c B$ arise from palettes in this way. Note that in the $n=0$ case, the only palette is $\pi = \emptyset$, and  $\pi^{\c B}=\sum_{i\in T}\g_i$, so there is exactly one green atom in $\c B$.
		
  Similarly, let $\c B'$ be the boolean sub-algebra of $\c B_{T', S}$ generated by the non-green atoms and $\set{b_0,\ldots, b_{n-1}}$, and let  $\pi^{\c B'}=\sum_{i\in \pi^{p_{n}}_{T'}}\g_i$.
   Since $\pi^{p_{n}}_T=\emptyset\iff \pi^{p_{n}}_{T'}=\emptyset$, for all palettes (by \eqref{t:dag2}), the map 
	\[\phi=\set{ (\pi^{\c B}, \pi^{\c B'}):\pi\subseteq\set{0, \ldots, n-1}}\setminus\set{(0,0)}\]
	is a bijection from the green atoms of $\c B$ to those of $\c B'$ which extends to a unique boolean isomorphism $\hat{\phi}:\c B\rightarrow\c B'$ fixing non-green atoms.

   We will show that $\c B$ is a relation algebra and that $\hat{\phi}$ is a relation algebra isomorphism. Note first that $\c B$ contains the identity and is closed under conversion, since all green elements are self-converse, furthermore both identity and converse are preserved by $\hat{\phi}$.  To show that $\c B$ is a relation algebra  and to show that $\hat\phi{\;}$ is a relation algebra isomorphism, we must show, for all atoms $x, y, z\in\c B$ that
   \begin{align}\label{eq:atom}
  \mbox{   either } (\times) (x;y)\cdot z=0 &\mbox{ or }(\checkmark) x;y\geq z, \mbox{ and}\\
 \label{eq:phi} (x\comp y)\cdot z = 0 &\iff (\phi(x)\comp \phi(y))\cdot \phi(z) =0.\end{align}
   In \eqref{eq:atom},  we indicate the  two alternatives by  $\times$ and $\checkmark$.

The cases where $1'\in\set{x, y, z}$ are easy, so suppose $x, y, z$ are non-identity atoms of $\c B$ (so each has a colour).  Recall that the only sets of three colours where some but not all triples of atoms of those colours are forbidden, are red-red-red and green-green-red  (rules \ref{f:red} and \ref{f:pim}).  In all other cases  either all triples of atoms of those colours are forbidden and we have $\times$ (and both sides of \eqref{eq:phi} are true) or none is forbidden and we have $\checkmark$ (and both sides are false).  If $z$ is red  we get $\checkmark$ or $\times$ automatically since $z$ is an atom of $\c B_{T, S}$. We also have \eqref{eq:phi} as, if $x$ and $y$ are red, then $\hat{\phi}$ fixes them, since they are non-green. Alternatively, if $x$ and $y$ are green then $x = \pi_1^{\c B}$ and $y=\pi_2^{\c B}$ for some palettes $\pi_1$ and $\pi_2$. In this case $(x\comp y)\cdot z \neq 0$ if and only if there is $i\neq i'\in T$ with $\g_i\leq x$ and $\g_{i'}\leq y$ (by rule \eqref{f:pim}), which occurs if and only if $\pi_1^{\c B}$ and $\pi_2^{\c B}$ are non-empty, if and only if $\pi_1^{\c B'}$ and $\pi_2^{\c B'}$ are non-empty (by the assumption that $\exists$ is using a winning strategy in $\G_{n}(T, T')$), if and only if $(\phi(x)\comp \phi(y))\cdot \phi(z) \neq 0$.   
If $x, y, z$ are all red, then they are fixed by $\hat\phi$, so \eqref{eq:phi} holds trivially.

  So, without loss of generality suppose that $x = \pi_1^{\c B}$ and $z=\pi_2^{\c B}$ are green and $y=\r_{j',j}$ is red (for some $j, j'\in S$). We are interested in $(\pi_1^{\c B}\comp \r_{j',j})\cdot \pi_2^{\c B}$. We want to show that either for every green atom $\g_i$ (of $\c B_{T, S}$) below $\pi_2^{\c B}$ there is a green atom $\g_{i'}$ below $\pi_1^{\c B}$ such that $(\g_{i'},\r_{j',j},\g_{i})$ is not forbidden (for $\checkmark$), or that for every green atom $\g_i$ (of $\c B_{T, S}$) below $\pi_2^{\c B}$ and for every green atom $\g_{i'}$ below $\pi_1^{\c B}$, the triple $(\g_{i'},\r_{j',j},\g_{i})$ is forbidden (for $\times$). 

Applying the Peircean equivalences, the triple under consideration here is equivalent to $(\g_i,\g_{i'},\r_{j,j'})$.  According to rule \eqref{f:pim}, this will be forbidden if and only if either $i=i'$, or $j=j'$. If $j=j'$, then $(\g_i,\g_{i'},\r_{j,j})$ is always forbidden (regardless of $i, i'$), so we have $\times$, and both sides of \eqref{eq:phi} are true. Alternatively, suppose $j\neq j'$ and let $\g_i\leq\pi_2^{\c B}$. If $x = z$, i.e. if $\pi_1=\pi_2$, then either $\pi_1^\c B=\pi_2^\c B=\g_i$ for some $i\in T$, in which case we have $\times$ as  $(\g_i,\g_{i},\r_{j,j'})$ is forbidden, or  $\pi_1^\c B=\pi_2^\c B$ is above at least two distinct green atoms, in which case we have $\checkmark$ as   $(\g_i,\g_{i'},\r_{j,j'})$ is not forbidden when $i\neq i'$ and $j\neq j'$.  As before, both sides of \eqref{eq:phi} are true, or they are both false in each case.  Similarly, if $x\neq z$, i.e. if $\pi_1\neq \pi_2$, then as $\pi_1^\c B\neq0$ and palettes are disjoint, for each green atom $\g_i$ below $\pi_2^\c B$, there must be $\g_{i'}\leq \pi_1^\c B$ distinct from $\g_i$. As $(\g_i,\g_{i'},\r_{j,j'})$ is not forbidden this proves $\checkmark$, and both sides of \eqref{eq:phi} are false. 
        
    It follows that $\c B$ is closed under all relation algebra operations, and is a sub-relation algebra of $\c B_{T, S}$.     Since $\set{a_0,\ldots,a_{n-1}}\subseteq\c B$ we have the inclusion of relation algebras,  $(\c B_{T, S} )_{\bar a} \subseteq \c B\subseteq \c B_{T, S}$.  Similarly, the boolean subalgebra $\c B'$ of $\c B_{T', S}$ generated by non-green atoms and  $\set{b_0,\ldots, b_{n-1}}$ is a sub-relation algebra of $\c B_{T', S}$ extending $(\c B_{T', S})_{\bar b}$.    By \eqref{eq:phi}, $\hat\phi{\;}$ is a relation algebra isomorphism from $\c B$ onto $\c B'$.

Moreover, for all $i<n$ we have $\hat{\phi}(a_i) = b_i$, as $\exists$'s strategy in $\Gamma_{n}(\c B_{T, S}, \c B_{T', S})$ ensures this is true. It follows that the restriction of $\hat{\phi}$ to $(\c B_{T, S})_{\bar a}$, which, as we have just proved, is an isomorphism onto $(\c B_{T', S})_{\bar b}$, is generated by $\set{(a_i, b_i):i<n}$. This proves the result. 
\end{proof}

\begin{corollary}\label{C:K-equiv}
Let $n<\omega$, let $T, T', S$ be finite sets.   If   $|T|, |T'| \geq 2^{n+1}$,  then $\c B_{T, S}\equiv_{n}\c B_{T', S}$.
\end{corollary}
\begin{proof}
By lemma \ref{L:K} $\exists$ has a winning strategy in $\G_{n}(T, T')$. So, by proposition \ref{P:K-equiv} she has a winning strategy in $\Gamma_{n}(\c B_{T, S}, \c B_{T', S})$. It follows by  lemma \ref{L:EF} that  $\c B_{T, S}\equiv_{n}\c B_{T',S}$.
\end{proof}

We can now prove our main result.

\begin{theorem}\label{thm:main}
If $\Sigma$ is a set of first-order formulas defining RRA, then $\Sigma$ includes formulas of arbitrary quantifier depth.
\end{theorem}
\begin{proof}
Let $T, S$ be sets where $|S|=2^{n+1},\; |T|=2^{n+1}+1$. Then $\c B_{S, S}\in RRA$ but $\c B_{T, S}\not\in RRA$, by theorem \ref{thm:rainbow}, and $\c B_{S, S}\equiv_{n}\c B_{T, S}$ by corollary \ref{C:K-equiv}.  It follows that RRA cannot be axiomatised by any theory consisting of sentences of quantifier depth at most $n$.
\end{proof}

\begin{remark}
This suggests that a similar construction could be used to prove that the class of representable cylindric algebras of dimension $n$ cannot be defined by a theory of bounded quantifier depth, however we have not succeeded in demonstrating this.    There is a way of constructing a \emph{rainbow cylindric algebra} of dimension $n\geq 4$ from two graphs $G, H$ given in \cite[\S4.3.3] {HH2}.  The atoms of this cylindric algebra are certain labelled hypergraphs on $n$ nodes.  The two-dimensional edges of these hypergraphs have green, red and white labels generalising the green, red, yellow and black atoms of  the rainbow relation algebra $\c A_{G, H}$, but these hypergraphs also have $(n-1)$-ary hyperlabels on some hyperedges in the cylindric version.    However, just as we were able to modify the rainbow relation algebra construction, essentially by deleting all atoms $\w_S$ for $|S|\leq 2$, we can modify the rainbow cylindric construction by deleting all hyperedges.
    It follows that $\c C^n_{G, H}$ is generated by its \emph{relation algebra reduct}.  By considering the graphs $K_{m+1}$ and $K_m$ we can show that $\c C^n_{K_{m+1}, K_m}$ is not in $RCA_n$ but $\c C^n_{K_m, K_m}$ is in $RCA_n$.       The problem is that although the atom structures of these two cylindric algebra agree on all $m$-variable atom structure formulas, we cannot prove  a cylindric version of Proposition~\ref{P:K-equiv},  so we do not   know if the two cylindric algebras are equivalent with respect to unrestricted formulas of quantifier depth at most $\log m$.  Thus our attempt to extend to various algebras of higher order relations using the known connections between relation algebras and cylindric algebras was not successful.
\end{remark}
%\bibliographystyle{abbrv}
%\bibliography{../robin}

\begin{thebibliography}{10}

\bibitem{AMN91}
H.~Andr{\'{e}}ka, J.~Monk, and I.~N{\'{e}}meti, editors.
\newblock {\em Algebraic logic}, volume~54 of {\em Colloq.\ Math.\ Soc.\ J.
  Bolyai}.
\newblock North-Holland, Amsterdam, 1991.

\bibitem{BW17}
R.~Berghammer and M.~Winter.
\newblock Solving computational tasks on finite topologies by means of relation
  algebra and the {R}el{V}iew tool.
\newblock {\em J. Log. Algebr. Methods Program.}, 88:1--25, 2017.

\bibitem{ChiTar51}
L.~H. Chin and A.~Tarski.
\newblock Distributive and modular laws in the arithmetic of relation algebras.
\newblock {\em Univ. California Publ. Math. (N.S.)}, 1:341--384, 1951.

\bibitem{EgrHirNote}
R.~Egrot and R.~Hirsch.
\newblock A corrected strategy for proving no finite variable axiomatisation
  exists for {RRA}.
\newblock https://arxiv.org/abs/2109.01357, 2021.

\bibitem{EgrHirgames}
R.~Egrot and R.~Hirsch.
\newblock Seurat games on {S}tockmeyer graphs.
\newblock {\em Journal of Graph Theory}, In Press.

\bibitem{FGPVGW16}
G.~H.~L. Fletcher, M.~Gyssens, J.~Paredaens, D.~Van~Gucht, and Y.~Wu.
\newblock Structural characterizations of the navigational expressiveness of
  relation algebras on a tree.
\newblock {\em J. Comput. System Sci.}, 82(2):229--259, 2016.

\bibitem{Giv88}
S.~Givant.
\newblock Tarski's development of logic and mathematics based on the calculus
  of relations.
\newblock In Andr{\'{e}}ka et~al. \cite{AMN91}, pages 189--215.

\bibitem{Giv06}
S.~Givant.
\newblock The calculus of relations as a foundation for mathematics.
\newblock {\em Journal of Automated Reasoning}, 37(4):277--322, 2006.

\bibitem{Gut18}
W.~Guttmann.
\newblock Verifying minimum spanning tree algorithms with {S}tone relation
  algebras.
\newblock {\em J. Log. Algebr. Methods Program.}, 101:132--150, 2018.

\bibitem{Hi94c}
R.~Hirsch.
\newblock Completely representable relation algebras.
\newblock {\em Bulletin of the interest group in propositional and predicate
  logics}, 3(1):77--92, 1995.

\bibitem{HH2}
R.~Hirsch and I.~Hodkinson.
\newblock Complete representations in algebraic logic.
\newblock {\em J. Symbolic Logic}, 62(3):816--847, 1997.

\bibitem{HH:raca2}
R.~Hirsch and I.~Hodkinson.
\newblock Relation algebras from cylindric algebras, {II}.
\newblock {\em Ann. Pure.\ Appl.\ Logic}, 112:267--297, 2001.

\bibitem{HH:book}
R.~Hirsch and I.~Hodkinson.
\newblock {\em Relation algebras by games}.
\newblock North-Holland. Elsevier Science, Amsterdam, 2002.

\bibitem{HHM98:jsl}
R.~Hirsch, I.~Hodkinson, and R.~Maddux.
\newblock Relation algebra reducts of cylindric algebras and an application to
  proof theory.
\newblock {\em J. Symbolic Logic}, 67(1):197--213, 2002.

\bibitem{HVCan}
I.~Hodkinson and Y.~Venema.
\newblock Canonical varieties with no canonical axiomatisation.
\newblock {\em Trans. Amer.\ Math.\ Soc.}, 357:4579--4605, 2005.

\bibitem{Im99}
N.~Immerman.
\newblock {\em Descriptive complexity}.
\newblock Graduate Texts in Computer Science. Springer-Verlag, New York, 1999.

\bibitem{Jon88}
B.~J{\'{o}}nsson.
\newblock The theory of binary relations.
\newblock In Andr{\'{e}}ka et~al. \cite{AMN91}, pages 245--292.

\bibitem{JT48}
B.~J{\'{o}}nsson and A.~Tarski.
\newblock Representation problems for relation algebras.
\newblock {\em Bull. Amer.\ Math.\ Soc.}, 54:80, 1192, 1948.

\bibitem{Ly50}
R.~Lyndon.
\newblock The representation of relational algebras.
\newblock {\em Annals of Mathematics}, 51(3):707--729, 1950.

\bibitem{Ma89}
R.~Maddux.
\newblock Non-finite axiomatizability results for cylindric and relation
  algebras.
\newblock {\em J. Symbolic Logic}, 54(3):951--974, 1989.

\bibitem{Ma91}
R.~Maddux.
\newblock The origin of relation algebras in the development and axiomatization
  of the calculus of relations.
\newblock {\em Studia Logica}, 50(3--4):421--455, 1991.

\bibitem{Mon64}
J.~Monk.
\newblock On representable relation algebras.
\newblock {\em Michigan Mathematics Journal}, 11:207--210, 1964.

\bibitem{T41}
A.~Tarski.
\newblock On the calculus of relations.
\newblock {\em J. Symbolic Logic}, 6:73--89, 1941.

\bibitem{T55}
A.~Tarski.
\newblock Contributions to the theory of models, {I, II}.
\newblock In {\em Proc.\ Konink.\ Nederl.\ Akad.\ van Wetensch.}, volume 57 (=
  Indag.\ Math.\ 16) of {\em A}, pages 572--581 and 582--588 resp., 1954.

\bibitem{TG87}
A.~Tarski and S.~Givant.
\newblock {\em A formalization of set theory without variables}.
\newblock Number~41 in Colloquium Publications. Amer.\ Math.\ Soc., Providence,
  Rhode Island, 1987.

\bibitem{V97:atom}
Y.~Venema.
\newblock Atom structures.
\newblock In M.~Kracht, M.~D. Rijke, H.~Wansing, and M.~Zakharyaschev, editors,
  {\em Advances in Modal Logic '96}, pages 291--305. CSLI Publications,
  Stanford, 1997.

\end{thebibliography}
 \def\www{/\allowbreak}

\end{document}